\documentclass[11pt]{article}    
\usepackage{amsmath,amssymb,amsthm,amsfonts,natbib}

\title{Outliers and patterns of outliers in contingency tables with Algebraic Statistics}

\author{Fabio Rapallo \\ Department DISTA, Universit\`a del Piemonte Orientale \\
Viale Teresa Michel, 11 \\
15121 Alessandria, Italy \\ \tt{fabio.rapallo@mfn.unipmn.it} }

\date{}

\begin{document}

\newtheorem{theorem}{Theorem}
\newtheorem{proposition}{Proposition}
\newtheorem{corollary}{Corollary}

\newcounter{example}
\newenvironment{example}
 {\refstepcounter{example}\paragraph {\it{Example \arabic{example}}}}
{\par
 \bigskip}

\newcounter{remark}
\newenvironment{remark}
 {\refstepcounter{remark}\paragraph {\it{Remark \arabic{remark}}}}
{\par
 \bigskip}

\newcounter{definition}
\newenvironment{definition}
 {\refstepcounter{definition}\paragraph {\bf{Definition \arabic{definition}}}}
{\par
 \bigskip}

\maketitle                 

\begin{abstract}
In this paper we provide a definition of pattern of outliers in
contingency tables within a model-based framework. In particular,
we make use of log-linear models and exact goodness-of-fit tests
to specify the notions of outlier and pattern of outliers. The
language and some techniques from Algebraic Statistics are
essential tools to make the definition clear and easily
applicable. We also analyze several numerical examples to show how
to use our definitions.

\bigskip

\end{abstract}

\section{Introduction} \label{intro}

The detection of outliers is one of the most important problems in
Statistics and it is a current research topic in the field of
contingency tables and categorical data. Some recent developments
in this direction can be found in \cite{kuhnt:04}, where the
author describes a procedure to identify outliers based on the
tails of the Poisson distribution and discusses the use of
different estimators to compute the expected counts under the null
hypothesis. A model-based approach to the detection of unexpected
cell counts is the Configural Frequency Analysis (CFA), where the
outlying counts are called ``types'' or ``antitypes'' if they are
significantly higher or smaller with respect to the expected
counts under a suitable model. The use of log-linear models for
CFA was presented in \cite{kieser|victor:99} and reanalyzed in
\cite{voneye|mair:08}. A complete account on theory and
applications of CFA can be found in \cite{voneye:02} and
\cite{voneye|mair|mun:10}.

The difficulties behind the definition of outlying cell in
contingency tables is proved by the number of different
approaches. About these difficulties, and more generally on the
old question: ``What a contingency table does say?'', an
interesting discussion is presented in
\cite{kateri|balakrishnan:08}. Some basic notions and appropriate
references for existing methods will be given later.

The notion of outlier for univariate and multivariate continuous
distributions is a well known fact. For example, in the univariate
case the outliers are usually detected through the boxplot or the
comparison of the standardized values with respect to the
quantiles of the normal distribution. It should be noted that
there is no unique mathematical definition of outlier, as pointed
out for instance in \cite{barnett|lewis:94}. Notice also that the
notion of outlier should be considered as outlier with respect to
a specified probability model. For instance, in the continuous
univariate case, it is usual to consider outliers with respect to
the Gaussian distribution, leading to the well known three-sigma
criterion.

The notion of outlier for contingency tables has a less clear
meaning. In fact, the random variables we consider are categorical
and the cells of the table are counts. When we consider
contingency tables, we do not define the outliers among the
subjects, but among the counts. As the counts can be modelled in a
simple Poisson sampling scheme, one would use the quantiles of the
Poisson distribution in order to detect the outliers in a
contingency table. Using a different approach, the detection of
outliers can also be deduced from the analysis of the adjusted
residuals. This approach has been presented in
\cite{fuchs|kenett:80} to test the presence of outliers in a
table, while the algorithm in \cite{simonoff:88} uses the adjusted
residuals and their contribution to the chi-squared Pearson's test
statistics to detect the position of the outlying cells.

In the past decade, Algebraic Statistics has been a very growing
research area, with major applications to the analysis of
contingency tables. Algebraic Statistics now provides an easy
description of complex log-linear models for multi-way tables and
it represents the natural environment to define statistical models
for contingency tables with structural zeros, through the notion
of toric models. Moreover, non-asymptotic inference is now more
actual via the use of Markov bases and the Diaconis-Sturmfels
algorithm. As general references on the use of Algebraic
Statistics for contingency tables, see
\cite{pistone|riccomagno|wynn:01}, \cite{pachter|sturmfels:05} and
\cite{drton|sturmfels|sullivant:09}. Some specific statistical
models to study complex structures in contingency tables can be
found in \cite{rapallo:05}, \cite{carlini|rapallo:10} and
\cite{carlini|rapallo:11}, with relevant applications in the
detection of special behaviours of some subsets of cells
(quasi-independence models, quasi-symmetry models, weakened
independence models).

In this paper, we use the dictionary, the reasoning and some
techniques from Algebraic Statistics in order to study the notion
of outliers in contingency tables. The outliers are defined in
terms of goodness-of-fit tests for tables with fixed cell counts.
Then, we investigate the main properties of the outliers and we
show how Algebraic Statistics is a useful tool both to make exact
inference for goodness-of-fit tests, and to easily describe
complex structures of outliers. We notice that the procedure
defined here is mainly useful as a confirmatory analysis after a
detection step based, for example, on the analysis of the
residuals. We will use this approach in the numerical examples,
detecting the candidate outliers through the residuals and then
testing them with the appropriate goodness-of-fit test. More
details on that issue will be discussed later in the paper.

The material is organized as follows. In Section \ref{recall-sect}
we recall some definitions and  basic results about toric models,
while in Section \ref{outliers-sect} we show how to study a single
outlying cell in the framework of toric models and we describe
explicitly the Monte Carlo test using Markov bases. In Section
\ref{sets-patterns-sect} we present the notions of sets and
patterns of outliers, and we analyze two real-data examples.
Finally, Section \ref{final-sect} contains some concluding remarks
and pointers to future works. In order to help readers with little
experience in polynomial algebra, we have decided to focus the
presentation on the statistical ideas. Thus, in the main body of
the paper we have avoided formal definitions whenever possible,
and we have grouped in the Appendix all the needed technical facts
from Algebraic Statistics.

\section{Some recalls about log-linear and toric models}
\label{recall-sect}

A probability distribution on a finite sample space ${\mathcal X}$
with $K$ elements is a normalized vector of $K$ non-negative real
numbers. Thus, the most general probability model is the simplex
\begin{equation*}
\Delta = \left\{(p_{1}, \ldots, p_{K}) \ : \ p_{k} \geq 0 \ , \
\sum_{k=1}^K p_{k} = 1   \right\} \, .
\end{equation*}
A statistical model ${\mathcal M}$ is therefore a subset of
$\Delta$.

A classical example of finite sample space is the case of a multi-way
contingency table where the cells are the joint counts of two or more
random variables with a finite number of levels each. In the case
of two-way contingency tables, where the sample space is usually
written as a cartesian product of the form ${\mathcal X}=\{1,
\ldots , I\} \times \{1 , \ldots, J \}$. We will consider this
case extensively in the next sections.

A wide class of statistical models for contingency tables are the
log-linear models \citep{agresti:02}. Under the classical Poisson
sampling scheme, the cell counts are independent and identically
distributed Poisson random variables with means $Np_1, \ldots,
Np_K$, where $N$ is the sample size, and the statistical model
specifies constraints on the parameters $p_1, \ldots, p_K$. A
model is log-linear if the log-probabilities lie in an affine
subspace of the vector space ${\mathbb R}^K$. Given $d$ real
parameters $\alpha_1, \ldots, \alpha_d$, a log-linear model is
described, apart from normalization, through the equations:
\begin{equation} \label{loglin}
\log (p_k) = \sum_{r=1}^d A_{k,r}\alpha_r
\end{equation}
for $k=1, \ldots, K$, where $A$ is the design matrix, see Ch.6 in \cite{pistone|riccomagno|wynn:01}. Exponentiating Eq.
\eqref{loglin}, we obtain the expression of the corresponding
toric model
\begin{equation} \label{toric}
p_k = \prod_{r=1}^d \zeta_r^{A_{k,r}}
\end{equation}
for $k=1, \ldots , K$, where $\zeta_{r} = \exp(\alpha_r)$, $r=1,
\ldots, d$, are the new non-negative parameters. It follows
immediately that the design matrix $A$ is also the matrix
representation of the minimal sufficient statistic of the model.

Notice that the model representations in Eq. \eqref{loglin} and
\eqref{toric} are equivalent on the open simplex, but the toric
representation allows us to consider also the boundary and,
therefore, the tables with structural zeros. This issue will be
essential in our definition of outliers. The matrix representation
of the toric models as in Eq. \eqref{toric} is widely discussed
in, e.g., \cite{rapallo:07} and
\cite{drton|sturmfels|sullivant:09}.

To obtain the implicit equations of the model, it is enough to
eliminate the $\zeta$ parameters from the system in Eq.
\eqref{toric}. In this paper, we will make use of the following
ingredients from Algebraic Statistics:
\begin{itemize}
\item[(i)] the toric ideal ${\mathcal I}_A$ of a statistical toric
model with design matrix $A$;

\item[(ii)] the variety ${\mathcal V}_A$ of the model;

\item[(iii)] the Markov basis ${\mathcal M}_A$ of the model.
\end{itemize}
To keep the exposition simple, we have collected the formal
definitions of these objects and some basic results on them in the
Appendix. We mention here only a few basic consequences of that
results that will be used in our presentation.

The toric ideal ${\mathcal I}_A$ of a toric model is by definition
the set of polynomials vanishing at each point of the model. Each
toric ideal is generated by a finite set of binomials, and thus we
can write
\begin{equation} \label{def-id}
{\mathcal I}_A = \mathrm{Ideal}(g_1, \ldots, g_\ell) \, ,
\end{equation}
meaning that each polynomial $g \in {\mathcal I}_A$ can be written
in the form $g = r_1g_1 + \ldots + r_\ell g_\ell$ for suitable
polynomials $r_1, \ldots , r_\ell$.

The binomials $g_1, \ldots, g_\ell$ can be actually computed with
symbolic software without any difficulties, at least for small-
and medium-sized tables, and we assume such binomials as given
together with the design matrix $A$. We write a binomial in
vectorial form $g=p^a-p^b$ meaning $g=\prod_k p_k^{a_k} - \prod_k
p_k^{b_k}$. Notice that for strictly positive probabilities the
equation $p^a-p^b=0$ is equivalent to $\log ({p^a} / {p^b})= 0$.
Therefore, the vanishing of a binomial correspond to the vanishing
of a log odds ratio and vice-versa. The vanishing log odds ratios
associated to a design matrix can be computed without polynomial
algebra, as they are the output of simple matrix computations.
Nevertheless, we emphasize that the usefulness of the binomials in
Definition \ref{def-id} is twofold:
\begin{itemize}
\item on one hand, the binomials $g_1, \ldots, g_\ell$ determine
the statistical model in the closed simplex $\Delta$. In fact, the
variety ${\mathcal V}_A$ associated to ${\mathcal I}_A$ is the set
of points
\begin{equation*}
{\mathcal V}_A = \left\{ p=(p_1, \ldots, p_K) \ : \ g_1(p) = 0 ,
\ldots, g_\ell(p) = 0 \right\} \subset {\mathbb R}^K
\end{equation*}
and, therefore, we obtain the statistical model simply by
normalization ${\mathcal V}_A \cap \Delta$;

\item on the other hand, the $\ell$ binomials naturally define
$\ell$ integer tables, called log-vectors, obtained by taking the
exponents of the $\ell$ binomials with the map
\begin{equation*}
g=p^a-p^b \longrightarrow m=a-b \, .
\end{equation*}
The tables $m_1, \ldots, m_\ell$ form a Markov basis ${\mathcal
M}_A$ for the model, which we will use to perform non-asymptotic
goodness-of-fit tests. See the Appendix for further details on
Markov bases.
\end{itemize}

To conclude, the binomials can be used both to study the geometry
of the statistical model and for the definition of a Markov basis
for the non-asymptotic goodness-of-fit test.

As an example in the two-way setting, the independence model for
$3 \times 3$ tables is represented by the matrix
\begin{equation*}
A_{\rm ind} = \left(\begin{matrix}1 & 1 & 0 & 1 & 0 \\
             1 & 1 & 0 & 0 & 1 \\
             1 & 1  & 0 & 0 & 0 \\
              1 & 0  & 1 & 1 & 0 \\
               1 & 0  & 1 & 0 & 1 \\
                1 & 0  & 1 & 0 & 0 \\
                1 & 0   & 0 & 1 & 0 \\
                 1 & 0  & 0 & 0 & 1 \\
                  1 &  0  & 0 & 0 & 0  \end{matrix}   \right) \, ,
\end{equation*}
while the quasi-independence model, which encodes independence of
the two random variables except for the diagonal cells is
represented by
\begin{equation*}
A_{\rm q-ind} = \left(\begin{matrix} 1 & 1& 0 & 1 & 0 & 1 & 0 & 0 \\
             1 & 1 & 0 & 0 & 1 & 0& 0 & 0 \\
             1 & 1  & 0 & 0 & 0 & 0& 0 & 0  \\
              1 & 0  & 1 & 1 & 0 & 0& 0 & 0  \\
               1 & 0  & 1 & 0 & 1 & 0& 1 & 0 \\
                1 & 0  & 1 & 0 & 0 & 1 & 0 & 0  \\
                1 &  0  & 0 & 1 & 0 & 0 & 0 & 0 \\
                 1 &  0  & 0 & 0 & 1 & 0 & 0& 0 \\
                  1 & 0  & 0 & 0 & 0 & 0& 0 & 1 \end{matrix} \right) \, .
\end{equation*}
The last three columns of $A_{\rm q-ind}$ force the diagonal cells
to be fitted exactly. For further details on the
quasi-independence models, see \cite{bishop|fienberg|holland:75}.
The equations of the independence model with design matrix $A_{\rm
ind}$ is the set of all $2 \times 2$ minors of the table of
probabilities, i.e.,
\begin{equation} \label{minors}
\begin{split}
{\mathcal I}_{A_{\rm ind}}= \mathrm{Ideal}
(p_{1,1}p_{2,2}-p_{1,2}p_{2,1}, \ p_{1,1}p_{2,3}-p_{1,3}p_{2,1},
\ p_{1,1}p_{3,2}-p_{1,2}p_{3,1}, \ \\
p_{1,1}p_{3,3}-p_{1,3}p_{3,1}, \ p_{1,2}p_{2,3}-p_{1,3}p_{2,2}, \
p_{1,2}p_{3,3}-p_{3,2}p_{2,3}, \  \\
p_{2,1}p_{3,2}-p_{3,1}p_{2,2}, \ p_{2,1}p_{3,3}-p_{3,1}p_{2,3}, \
p_{2,2}p_{3,3}-p_{3,2}p_{2,3} ) \, ,
\end{split}
\end{equation}
while for the quasi-independence model from the matrix $A_{\rm
q-ind}$ we have only one binomial:
\begin{equation*}
{\mathcal I}_{A_{\rm q-ind}}= \mathrm{Ideal}
(p_{1,2}p_{2,3}p_{3,1}-p_{1,3}p_{3,2}p_{2,1} ) \, .
\end{equation*}

\begin{remark}
We point out that the independence model can be described in terms
of vector spaces by $4$ linearly independent log-vectors
\citep{agresti:02}, and typically one can use the log-vectors of
the $4$ adjacent minors. but to have a Markov basis we need all
the $9$ binomials in Eq. \eqref{minors}.
\end{remark}

Notice that, from the point of view of the statistical models, a
fixed cell count has the same behaviour as a structural zero. See
\cite{rapallo:06} for a discussion on this issue. This fact
suggests that outliers can be modelled in the framework of
statistical models with structural zeros, as we will make precise
in the following section. The use of structural zeros to model
contingency tables with complex structure is presented in
\cite{consonni|pistone:07} under the point of view of Bayesian
inference.

\begin{remark}
In the special case of independence model for two-way tables, the
use of $2 \times 2$ minors as in Eq. \eqref{minors} to detect
outliers was implemented in \cite{kotze|hawkins:84}. We also
mention that the connections between the implicit equations of the
model and the adjusted residuals are known at least in the simple
case of the independence model for two-way table, see for instance
\cite{tsumoto|hirano:07}.
\end{remark}

\section{Outliers} \label{outliers-sect}

\begin{example} \label{ex-synt}
Let us consider the following synthetic contingency table:
\begin{equation} \label{synth-ex}
f = \left( \begin{matrix}7 & 2&  2&  2 \\
                  2 & 2 & 2&  2 \\
                  2 & 2&  2 & 2 \\
                  3 & 2 & 2 & 2 \end{matrix} \right) \, .
\end{equation}
Under the independence model, it seems that the cell $(1,1)$ could
be an outlier.

With the approach presented in \cite{fuchs|kenett:80}, the
observed contingency table $f$ is the realization of a multinomial
distribution and the authors analyze the adjusted residuals under
the independence model
\begin{equation*}
Z_{i,j} = \frac {f_{i,j} - f_{i,+}f_{+,j}/N} {\sqrt{
f_{i,+}(N-f_{i,+})f_{+,j}(N-f_{+,j})/N^3}}
\end{equation*}
for $i=1, \ldots, I$ and $j=1, \ldots, J$, where $N$ is the sample
size and $f_{i,+}$ and $f_{+,j}$ are the row and column sums,
respectively. To check the presence of outlying cells, the authors
use the test statistics $Z=\max_{i,j} |Z_{i,j}|$ and they find
suitable approximations for the two-sided $\alpha$-level critical
value, using the standard Normal distribution. The use of the
adjusted residuals to detect outliers was first described in
\cite{haberman:73}. However, we warn that the test in
\cite{fuchs|kenett:80} is a global test and it is not useful to
detect the position of the outliers in the table.

On the other hand, the approach described in \cite{kuhnt:04} is
based on the computation of the ML (or $L_1$) estimate of the
mean of the Poisson distributions for the cell counts, and then a
cell is declared as outlier if the actual count falls in the tails
of the appropriate Poisson distribution.

Let us analyze the observed table $f$ above under the two
approaches described here. Using the adjusted residuals as in
\cite{fuchs|kenett:80}, the value of the test statistics is
$z=1.5670$ (the highest adjusted residuals), while the critical
value at the $\alpha=5\%$ level is $2.9478$, showing that there is
no evidence of outlying cells. Under the Poisson approach as in
\cite{kuhnt:04}, we find that the observed value in the cell
$(1,1)$ is not considered an outlier at the $5\%$-level, both
using the standard ML estimate $\hat f_{1,1}=4.7895$ (outlier
region $[9, +\infty)$), and using the more robust $L_1$ estimate
$\tilde f_{1,1}=3.5$ (outlier region $[8,+\infty)$).
\end{example}

As mentioned above, we adopt here a different
point of view to set up the definition and the detection of the
outliers in a contingency table. We define them using a
model-based approach with appropriate goodness-of-fit tests for
the comparison of two nested models. The starting point is similar
to the definition of types and antitypes in CFA, see
\cite{kieser|victor:99}, but after the first definitions we will
use Algebraic Statistics to understand and generalize the notion
of outlier.

Given a contingency table with $K$ cells, let us consider a
statistical toric model for the table. The model has the
expression:
\begin{equation} \label{base}
p_k = \prod_{r=1}^d \zeta_r^{A_{k,r}}
\end{equation}
for all $k = 1, \ldots, K$. This model with matrix representation
$A$ will be named as the base model. Moreover, let $\alpha \in
(0,1)$.

\begin{definition} \label{one-out}
The cell $h$, $h \in \{1, \ldots, K\}$ is an $\alpha$-level outlier
with respect to the base model if the model
\begin{equation} \label{with-out}
p_k = \left\{ \begin{array}{lll} \prod_r \zeta_r^{A_{k,r}} & \ \ &
\mbox{
for } \ k \ne h \\ \\
\prod_r \zeta_r^{A_{h,r}} \zeta^{(s)}_{h} & \ \ & \mbox{ for } \ k
= h \end{array} \right.
\end{equation}
is significantly better than the base model at level $\alpha$,
where $\zeta^{(s)}_{h}$ is a new non-negative parameter.
\end{definition}

This means that we compare two toric models:
\begin{itemize}
\item the base model in Eq. \eqref{base} with matrix
representation $A$;

\item the model in Eq. \eqref{with-out}, whose design matrix is
\begin{equation*}
\tilde A = [ A \ | \ I_h ]
\end{equation*}
where $I_h$ is the indicator vector of the cell $h$: $I_h$ is a
vector of length $K$ with all components equal to $0$ but the
$h$-th component equal to $1$.
\end{itemize}

Notice that we do not test the goodness-of-fit of the model in Eq.
\eqref{with-out}, but we only compare it with the base model.

To avoid trivialities in Definition \ref{one-out}, we suppose that
the cell $h$ is not a component of the sufficient statistic of the
base model, i.e., we suppose that the matrices $A$ and $\tilde A$
satisfy the relation: $\mathrm{rank}(\tilde A) =
\mathrm{rank}(A)+1$. In fact, if $\mathrm{rank}(\tilde A) =
\mathrm{rank}(A)$, then the count in the cell $h$ is already a
component of the sufficient statistic of the base model and the
goodness-of-fit test becomes useless.

From the point of view of toric models, the new parameter
$\zeta^{(s)}_{h}$ imposes the exact fit of the candidate outlier
$h$. Although it is possible to find easy algebraic relations
between the ideal ${\mathcal I}_A$ of the base model and the ideal
${\mathcal I}_{\tilde A}$, we focus here on the geometric analysis
of the statistical models. In terms of varieties, the variety
${\mathcal V}_{A}$ is a subset of ${\mathcal V}_{\tilde A}$. This
follows from the proposition below. We will use it also in the
next section, thus we state the result in a general setting.

\begin{theorem} \label{propincl}
Let $A_1$ and $A_2$ be two integer non-negative matrices with $K$
rows, and let $\mathrm{Im}(A_1)$ and $\mathrm{Im}(A_2)$ be their
images, as vector spaces in ${\mathbb R}^K$. If $\mathrm{Im}(A_1)
\subset \mathrm{Im}(A_2)$, then ${\mathcal V}_{A_1} \subset
{\mathcal V}_{A_2}$.
\end{theorem}
\begin{proof}
By virtue of Proposition \ref{dualincl} in the Appendix, we have
to show that ${\mathcal I}_{A_2} \subset {\mathcal I}_{A_1}$. Let
$g$ be a polynomial in ${\mathcal I}_{A_2}$. Then,
\begin{equation*}
g = r_1g_1 + \ldots + r_\ell g_\ell
\end{equation*}
where $\{g_1, \ldots, g_\ell\}$ is a system of generators of
${\mathcal I}_{A_2}$ and $r_1, \ldots, r_\ell$ are polynomials.

From Theorem \ref{DS-teo} in the Appendix, $g_1, \ldots, g_\ell$
are binomials and their log-vectors (see Definition
\ref{logvector} in the Appendix) $m_1, \ldots, m_\ell$ are in
$\ker(A^t_2)$. As $\ker(A^t_2) \subset \ker(A^t_1)$, we have also
that $g \in {\mathcal I}_{A_1}$. This proves the result.
\end{proof}

The inclusion ${\mathcal V}_{A} \subset {\mathcal V}_{\tilde
A}$ follows from Theorem \ref{propincl} with $A_1=A$ and $A_2
= \tilde A$.

To actually check if a cell is an outlier, it is enough to
implement the goodness-of-fit test in Definition \ref{one-out}.
This test can be done using the log-likelihood ratio statistic
\citep[page 591]{agresti:02}. The test statistic has the
expression
\begin{equation*}
G^2 = 2 \sum_{k=1}^K f_k \log \left( \frac {\hat f_{1k}} {\hat
f_{0k}} \right) \, ,
\end{equation*}
where $\hat f_{0k}$ and $\hat f_{1k}$ are the maximum likelihood
estimates of the expected cell counts under the base model with design
matrix $A$ and the model with design matrix $\tilde A$,
respectively. The value of $G^2$ must be compared with the
appropriate quantiles of the chi-square distribution with $1$ df.

Alternatively one can make exact inference via Markov bases and
the Diaconis-Sturmfels algorithm (see Ch.1 in \cite{drton|sturmfels|sullivant:09}).

Given an observed contingency table $f \in {\mathbb N}^K$ and a
Markov basis ${\mathcal M}_A$ for the base model, one can apply
the Diaconis-Sturmfels algorithm by sampling $B$ contingency
tables from its reference set
\begin{equation*}
{\mathcal F}_A(f) = \left\{ f' \in {\mathbb N}^K \ : \ A^t f' = A^t f
\right\} \, .
\end{equation*}
The reference set is the set of all contingency tables with the
same value of the sufficient statistic $A^t f$ as the observed
table. The relevant distribution on ${\mathcal F}_A(f)$ is the
hypergeometric distribution ${\mathcal H}(f')$, and the explicit
expression of this distribution is
\begin{equation*}
{\mathcal H}(f') = \frac {1 / \prod_k 1/(f'_k)!} {\sum_{f^*\in
{\mathcal F}_A(f)} 1 / \prod_k 1/(f^*_k)! } \, .
\end{equation*}
See \cite{drton|sturmfels|sullivant:09} for details on the
derivation of this distribution. To actually sample from the
reference set with the prescribed distribution, we implement a
Metropolis-Hastings Markov chain starting from the observed table.
At each step:
\begin{enumerate}
\item let $f$ be the current table;

\item choose with uniform probability a move $m \in {\mathcal
M}_A$ and a sign $\epsilon= \pm 1$ with probability $1/2$ each;

\item define the candidate table as $f_+=f+\epsilon m$;

\item generate a random number $u$ with uniform distribution over
$[0,1]$. If $f_+ \geq 0$ and
\begin{equation*}
\min \left\{ 1 , \frac {\mathcal H(f_+)}  {\mathcal H(f)} \right\}
> u
\end{equation*}
then move the chain in $f_+$; otherwise stay at $f$.
\end{enumerate}
The use of a Markov basis as set of moves ensures the
connectedness of the Markov chain. The proportion of sampled
tables with test statistics greater than or equal to the test
statistic of the observed one is the Monte Carlo approximation of
$p$-value of the log-likelihood ratio test.

\begin{example}
Analyzing the contingency table in Example \ref{ex-synt} with a
Monte Carlo approximation based on $B=10,000$ tables we obtain an
approximated $p$-value $0.1574$, showing that there is no evidence
to conclude that the cell $(1,1)$ is an outlier. In this example,
the asymptotic $p$-value based on the chi-squared approximation is
$0.0977$, with a noteworthy difference with respect to the Monte
Carlo approach. Notice that in similar problems the asymptotic
approximation dramatically fails. To see this, consider the
observed table

\begin{equation*}
f' = \left( \begin{matrix}0 & 2&  2&  2 \\
                  2 & 2 & 2&  2 \\
                  2 & 2&  2 & 2 \\
                  3 & 2 & 2 & 2 \end{matrix} \right) \, .
\end{equation*}
This table differs from the first example in Eq. \eqref{synth-ex}
only in the first cell. Here, the cell $(1,1)$ is an antitype with
an observed count less than the expected under independence, while
in Eq. \eqref{synth-ex} the cell $(1,1)$ was a type. For this
table $f'$, the Monte Carlo $p$-value is $0.1856$, while the
corresponding asymptotic approximation is $0.0522$.
\end{example}

All the simulations presented in this paper has been performed in
{\tt R}, see \cite{rproject:10} together with the {\tt gllm}
package to make inference on generalized log-linear models
\citep{duffy:10}.

\begin{remark}
From the discussion in Example \ref{ex-synt} one sees that we have used our procedure only for the confirmatory step. Nevertheless, in the simple case of a single outlier the test can also be used to detect an outlier. Is is enough to run the test once for each cell.
\end{remark}

Finally, we remark that in many cases the computation of a Markov
basis ${\mathcal M}_A$ for the base model does not need explicit
symbolic computations. In fact, for several statistical models,
such as independence, symmetry, quasi-independence, a Markov basis
has been computed theoretically, see
\cite{drton|sturmfels|sullivant:09} and \cite{rapallo:03}. For
instance, our numerical example in this section considers the
independence model as base model and a suitable Markov basis is
formed by the $36$ basic moves of the form $\begin{pmatrix} +1 &
-1 \\ -1 & +1\end{pmatrix}$ for all $2 \times 2$ minors of the
table.

In view of the connections between Markov bases and varieties,
this example is quite simple from the point of view of Geometry.
In fact, the variety of the base model is described by the
vanishing of all $2 \times 2$ minors of the table of
probabilities. In the same way, it is easy to see that the variety
of the model with one outlier is described by the vanishing of the
$27$ $2 \times 2$ minors not involving the $(1,1)$ cell.

\section{Sets and patterns of outliers} \label{sets-patterns-sect}

Definition \ref{one-out} can be easily extended to a set of
outliers.

\begin{definition} \label{set-out}
The cells $h_1 , \ldots, h_m$ form an $\alpha$-level set of
outliers with respect to the base model if the model
\begin{equation} \label{with-set-out}
p_k = \left\{ \begin{array}{lll} \prod_r \zeta_r^{A_{k,r}} & \ \
 & \mbox{ for } \ k \ne h_1, \ldots, h_m \\
\\
 \prod_r \zeta_r^{A_{k,r}} \zeta^{(s)}_{k} & \ \ & \mbox{ for } \
k=h_1, \ldots, h_m \end{array} \right.
\end{equation}
is significantly better than the base model at level $\alpha$,
where $\zeta^{(s)}_{h_1}, \ldots, \zeta^{(s)}_{h_m}$ are $m$ new
non-negative parameters.
\end{definition}

In analogy with our previous analysis, notice that the model in
Eq. \eqref{with-set-out} has matrix representation
\begin{equation*}
\tilde A = [ A \ | \ I_{h_1} \ | \ \cdots \ | \ I_{h_m} ] \, ,
\end{equation*}
where $I_{h_1}, \ldots, I_{h_m}$ are the indicator vectors of the
cell $h_1, \ldots, h_m$ respectively.

Also in this definition, to avoid trivialities, we suppose that
the cells $h_1, \ldots, h_m$ are not components of the sufficient
statistic of the base model, i.e., we suppose that
$\mathrm{rank}(\tilde A) > \mathrm{rank}(A)$. It is clear that the
difference $\mathrm{rank}(\tilde A) - \mathrm{rank}(A)$ is just
the number of degrees of freedom of the goodness-of-fit test. The
test procedure can be performed with the same technique as for a
single outlier. The algorithm is essentially the same as in
Section \ref{outliers-sect} for a single outlier.

\begin{example}
Let us consider the independence model for $4 \times 4$ tables as
the base model, as in the previous discussion. Now, we look at the
$8$ cells on the diagonal and the anti-diagonal as the set of
outliers. The ideal of the base model is generated by the $36$ $2
\times 2$ minors of the table of probabilities, while computation
of the ideal without the $8$ variables $p_{1,1}, \ldots,
p_{4,4},p_{1,4}, \ldots, p_{4,1}$ gives an ideal generated by the
$2$ binomials:
\begin{equation*}
-p_{1,3}p_{4,2} + p_{1,2}p_{4,3}, \ -p_{2,4}p_{3,1} +
p_{2,1}p_{3,4} \ .
\end{equation*}

When the dimensions of the table increase, the toric ideals become
more complicated. For instance, the same problem as above for $5
\times 5$ tables yields a base model generated by the $100$ $2
\times 2$ minors of the table of probabilities, and the toric
ideal without the $9$ variables $p_{1,1}, \ldots, p_{5,5},p_{1,5},
\ldots, p_{5,1}$ is generated by $28$ binomials: $10$ binomials of
degree $2$ of the form $- p_{1,4 }p_{3,2} + p_{1,2}p_{3,4}$, and
$18$ binomials of degree $3$ of the form $p_{3, 5 }p_{4, 3 }p_{5,
2 } - p_{3, 2 }p_{4, 5 }p_{5, 3}$.
\end{example}

As mentioned in the Introduction, one among the key points of
Algebraic Statistics lies in the possibility to make the
description and the meaning of log-linear models easier. Thus, we
can enrich the base model in many ways.

\begin{definition} \label{pattern-out}
The cells $h_1 , \ldots, h_m$ form an $\alpha$-level pattern of outliers with
respect to the base model if the model
\begin{equation*}
p_k = \left\{ \begin{array}{lll} \prod_r \zeta_r^{A_{k,r}} & \ \ &
\mbox{ for } \ k \ne h_1, \ldots, h_m \\
\\
\prod_r \zeta_r^{A_{k,r}} \zeta^{(p)} & \ \  & \mbox{ for } \ k =
h_1, \ldots, h_m
\end{array} \right.
\end{equation*}
is significantly better than the base model, where $\zeta^{(p)}$
is a new non-negative parameter.
\end{definition}

To avoid trivialities in Definition \ref{pattern-out}, we suppose
that the indicator vector of the cells $h_1, \ldots, h_m$ is not a
component of the sufficient statistic of the base model, i.e., we
suppose that the matrices $\tilde A$ and $A$ satisfy:
$\mathrm{rank}(\tilde A) = \mathrm{rank}(A)+1$.

\begin{remark}
Notice that in Definition \ref{pattern-out} the outlying cells in
a pattern are characterized by a single parameter $\zeta^{(p)}$.
This means that we assume a common behaviour of that cells.
\end{remark}

As an immediate consequence of Theorem \ref{propincl}, we have the
following result about the connections between sets and patterns
of outliers.

\begin{proposition}
Let $h_1, \ldots, h_m$ be $m$ cells. The model with $h_1, \ldots,
h_m$ as a set of outliers contains the model with  $h_1, \ldots,
h_m$ as a pattern of outliers.
\end{proposition}

It follows that the definition of set of outliers in Definition
\ref{set-out} is stronger than the definition of pattern of
outliers. On the other hand, the notion of pattern of outliers may
help in finding parsimonious models.

\begin{remark}
In the case of sets and patterns of outliers, the procedure
presented in this paper is confirmatory, and a preliminary step is
needed in order to select the potential outliers. This step can be
done through the analysis of the residuals under the base model.
We follow this approach in the numerical examples below.
\end{remark}

\begin{example}
The definitions of set of outliers and pattern of outliers are
very flexible and can be combined in many ways. In order to show
this feature, we reconsider the following data analyzed in
\cite{voneye|mair:08} about the size of social network. The sample
is formed by $516$ individuals, classified by marital status
($M=1$ married, $M=2$ not married), gender ($G=1$ male; $G=2$
female), and size of social network ($S=1$ small, $S=2$ large).
The $8$ cell counts are listed in Table \ref{dati}, together with
the expected cell counts $\hat f$ and the Pearsonian residuals $(f
- \hat f)/\hat f)$.

\begin{table}[ht]
\begin{center}
\begin{tabular}{ccc|c|c|c}
$M$ & $G$ & $S$ & $f$ & $\hat f$ & $(f-\hat f)/\hat f$ \\ \hline
$1$ & $1$ & $1$ & $48$ & $38.9$ & $1.45$ \\
$1$ & $1$ & $2$ & $87$ & $38.1$ & $7.93$ \\
$1$ & $2$ & $1$ & $5$ & $38.9$ & $-5.44$ \\
$1$ & $2$ & $2$ & $14$ & $38.1$ & $-3.90$ \\
$2$ & $1$ & $1$ & $78$ & $91.6$ & $-1.42$ \\
$2$ & $1$ & $2$ & $45$ & $89.4$ & $-4.70$ \\
$2$ & $2$ & $1$ & $130$ & $91.6$ & $4.02$ \\
$2$ & $2$ & $2$ & $109$ & $89.4$ & $2.07$
\end{tabular}
\end{center}
\caption{Data on social network size.}  \label{dati}
\end{table}

As a base model, we use the complete independence model, which can
be written in log-linear form (with the usual log-linear notation)
as:
\begin{equation*}
\log p_{i,j,k} = \lambda + \lambda_i^{(M)} + \lambda_j^{(G)} +
\lambda_k^{(S)} \, .
\end{equation*}
The ideal of this base model is:
\begin{equation*}
\begin{split}
\mathrm{Ideal}(p_{1,2,1}p_{2,1,1} - p_{1,1,1}p_{2,2,1},
p_{1,2,1}p_{2,1,2} - p_{1,1,2}p_{2,2,1},  \\
-p_{1,2,2}p_{2,2,1} +
p_{1,2,1}p_{2,2,2}, -p_{2,1,2}p_{2,2,1} + p_{2,1,1}p_{2,2,2}, \\
-p_{1,1,2}p_{2,1,1} + p_{1,1,1}p_{2,1,2}, p_{1,2,2}p_{2,1,1} -
p_{1,1,2}p_{2,2,1},  \\
p_{1,2,2}p_{2,1,2} - p_{1,1,2}p_{2,2,2},
-p_{1,1,2}p_{2,2,1} + p_{1,1,1}p_{2,2,2},  \\
-p_{1,1,2}p_{1,2,1} +
p_{1,1,1}p_{1,2,2}) \, .
\end{split}
\end{equation*}
Thus, a Markov basis for this model is formed by $9$ moves. A
quick inspection of the residuals suggests that the cells
$(1,1,2)$ and $(2,2,1)$ are potential types, while the cells
$(1,2,1)$, $(1,2,2)$ and $(2,1,2)$ are potential antitypes.

If one would run a test for each of the previous cells as in
Definition \ref{one-out}, the approximated Monte Carlo $p$-values
are $0$ in all cases. Notice also that in this example the
definition of set of outliers as in Definition \ref{set-out} is
not helpful, as the corresponding model become saturated. However,
if we run the Monte Carlo test as in Definition \ref{pattern-out}
with these $5$ cells as a unique pattern of outliers, we obtain a
$p$-value $0.1411$, showing that the $5$ cells do not have a
common behaviour, but the test with two patterns of outliers,
namely the potential types and antitypes separately, exhibits a
$p$-value $0.0001$, with strong evidence that the cells in the two
patterns $\{(1,1,2),(2,2,1)\}$ and $\{(1,2,1),(1,2,2),(2,1,2) \}$
have a homogeneous behaviour in deviating from the base model. The
design matrix for this model is
\begin{equation*}
\tilde A = \begin{pmatrix}
1 & 1 & 1 & 1 & 0 & 0 \\
1 & 1 & 1 & 0 & 1 & 0 \\
1 & 1 & 0 & 1 & 0 & 1 \\
1 & 1 & 0 & 0 & 0 & 1 \\
1 & 0 & 1 & 1 & 0 & 0 \\
1 & 0 & 1 & 0 & 0 & 1 \\
1 & 0 & 0 & 1 & 1 & 0 \\
1 & 0 & 0 & 0 & 0 & 0 \\
\end{pmatrix} \, ,
\end{equation*}
where the first $4$ columns of $\tilde A$ correspond to the
parameters of the base model, while the last two columns
correspond to the two parameters additional parameters of the
model with two patterns of outliers. In this example, we are able
to describe the outlying cells with only two additional
parameters. The interpretation of this model could be that the
three types and two antitypes have common causes, but such an
interpretation would require a more detailed data analysis and is
beyond the scope of this paper. Here, we limit ourselves to
provide a mathematical description of the outliers.

We note that the model with two patterns of outliers has a less
clear geometric description with respect to the base model. In
fact, the corresponding ideal is:
\begin{equation*}
\begin{split}
\mathrm{Ideal}(-p_{1,2,2}^2p_{2,1,1}^2 +
p_{1,1,1}p_{1,2,1}p_{2,1,2}p_{2,2,2}, \\
-p_{1,1,2}p_{1,2,1}p_{2,1,1}^2 + p_{1,1,1}^2p_{2,1,2}p_{2,2,1},
p_{1,1,1}p_{1,2,2}^2p_{2,2,1} - p_{1,1,2}p_{1,2,1}^2p_{2,2,2}, \\
p_{1,2,2}^4p_{2,1,1}^2p_{2,2,1} -
p_{1,1,2}p_{1,2,1}^3p_{2,1,2}p_{2,2,2}^2) \, .
\end{split}
\end{equation*}
\end{example}

\begin{example} \label{big-ex}
In this example, we show the practical applicability of our
technique in the case of large tables. We analyze the data
presented in \cite{agresti:02} as an exercise on logit models for
multinomial responses. The contingency table, reported in Table
\ref{sec-tab}, refers to a sample of residents of Copenhagen. The
individuals of the sample were classified according to $4$
categorical variables: type of housing ($H$), degree of contact
with other residents ($C$), feeling of influence on apartment
management ($I$), and satisfaction with housing conditions ($S$).
The table has dimensions $4 \times 3 \times 2 \times 3$, for a
total of $72$ cells, and $S$ has the role of response variable.

\begin{table}[ht]
\begin{center}
\begin{tabular}{cc|ccc||ccc}
      &   Contact      &             &  Low   &       &        &  High  &    \\  \cline{2-8}
        &   Satisfaction &        Low &  Medium  & High &         Low &  Medium  & High \\  \hline
  Housing      &        Influence &          &          &      &             &         & \\ \hline
Tower blocks  &     Low  &    21  &     21  &     28   &       14 &    19  &   37  \\
              &    Medium &  {\bf 34}  &    22  &     36   &        17 &    23 &    40 \\
              &      High  &   10  &    11   &    36     &        3  &    5   &  23  \\ \hline
Apartments    &       Low &    61  &    23  &    17    &       78 &   46 &   43  \\                          &   Medium  &   43   &   35   &    40    &       48 &   45 &    86 \\
              &       High  &   26 &     18 &      54 &          15 &   25  &   62 \\ \hline    Atrium houses &       Low  &   13  &     9  &     10  &         20  &   23  &   20 \\
              &    Medium  &    8  &    8   &    12  &         10   &  22  &  24 \\
              &      High  &    6 &     7  &      9  &          7  &   10  &  21 \\ \hline
Terraced houses   &   Low  &   18  &     6  &      7     &    {\bf 57} &    23 &    13 \\
                &  Medium  &   15   &   13  &    13     &      31  &   21  &   13 \\
                 &   High   &   7     & 5   &    11   &         5    &  6 &    13 \\ \hline
\end{tabular}
\end{center}
\caption{Data on housing conditions in Copenhagen.}
\label{sec-tab}
\end{table}

As base model, we use a log-linear model including the $4$ main
effects and the interactions $[HS], [CS], [IS]$, that is, the
interactions between the response variable and the other three
variables. This model has $51$ degrees of freedom and fits the
data poorly. A Markov basis for this model is formed by $360$
moves and its computation with {\tt 4ti2} in carried out in few
seconds.

Analyzing the residuals of this table under the base model, we
note that there are $2$ Pearsonian residuals exceeding $3$ (in
absolute value). The two cells are:
\begin{itemize}
\item[-] $H=$``Tower blocks'', $C=$``Low'', $I=$``Medium'',
$S=$``Low''. The observed count is $34$ versus a predicted count
$16.62$, with a Pearsonian residual equal to $4.263$;

\item[-] $H=$``Terraced houses'', $C=$``High'', $I=$``Low'',
$S=$``Low''. The observed count is $57$ versus a predicted count
$35.58$, with a Pearsonian residual equal to $3.590$.
\end{itemize}
(the counts of these cells are printed in bold in Table
\ref{sec-tab}).

We consider these two cells as a set of outliers and we run the
Monte Carlo algorithm as in the previous example. The approximated
Monte Carlo $p$-value is $0$ (and the asymptotic $p$-value is $1.8
\cdot 10^{-9}$). This shows that the proposed set of outliers is
highly significant. Moreover, we note that the log-likelihood
ratio statistic decreases from the value of $123.19$ for the base
model to $88.51$ for the outlier model adding only $2$ parameters.
Looking at the table, this means that these two cells have a
special behaviour, and a particular inspection of the above
combinations could give relevant information on the data.
\end{example}

\section{Final remarks} \label{final-sect}

In this paper, we have shown how Algebraic Statistics is useful in
addressing the problem of outliers in contingency tables. In
particular, we have shown the efficacy of this approach in two
directions: (a) the use of non-asymptotic inference for
statistical models to recognize outliers; (b) a simple and
practical description of such statistical models from the point of
view of Geometry.

In particular, we have shown that Algebraic Statistics allows us
to a simple definition of set of outliers, patterns of outliers,
and their combinations.

Of course, the theory presented here does not exhaust all the
research themes on this topic. Many questions remain still open,
and among these problems we mention: the need for procedures and
algorithms for the recognition of outliers; the problems of the
choice of the $\alpha$-level for multiple tests, using
Bonferroni-type techniques. These problems are widely discussed in
many articles cited above, see e.g. \cite{kieser|victor:99}.

From the perspective of Algebraic Statistics, some interesting
issues are yet to be explored:
\begin{itemize}
\item The connections between the models studied here and the
mixture models. Mixture models for the special case of outliers on
the main diagonal are already considered in
\cite{bocci|carlini|rapallo:10};

\item The characterization of the Markov bases for the models with
outliers can yield useful information about the structure of the
corresponding statistical models. Although in the case of a single
pattern of outliers some Markov bases are already computed in
\cite{hara|takemura|yoshida:09}, yet the general case with several
outliers and patterns of outliers is currently unexplored.
\end{itemize}

\section*{Acknowledgments}

We acknowledge the help and support of Enrico Carlini (Politecnico
di Torino, Italy), who has provided several suggestions for a
precise and clear algebraic presentation. We also thank the
anonymous referees and the Associate Editor for their valuable
suggestions to improve the quality and the readability of the
paper.

\appendix

\section{Basic definitions and tools from Algebraic Statistics} \label{app-basic}

In this appendix we collect some basic facts about toric ideals
and statistical toric models. A more detailed presentation of
these results can be found in \cite{drton|sturmfels|sullivant:09}.
For some basic algebraic definitions we also refer to
\cite{pistone|riccomagno|wynn:01}.

Let ${\mathbb R}[p,\zeta]={\mathbb R}[p_1, \ldots , p_K, \zeta_0,
\zeta_1, \ldots, \zeta_d]$ be the polynomial ring in the variables
$p_1, \ldots , p_K, \zeta_1, \ldots, \zeta_d$ with real
coefficients.

\begin{definition}[Polynomial ideal] \label{def-ideal}
An ideal ${\mathcal I}$ in ${\mathbb R}[p,\zeta]$ is a set of
polynomials such that for all $g, h \in {\mathcal I}$, $g + h \in
{\mathcal I}$ and for all $g \in {\mathcal I}, h \in {\mathbb
R}[p,\zeta]$, $gh \in {\mathcal I}$.
\end{definition}

The Hilbert's basis theorem states that every polynomial ideal
${\mathcal I}$ as in Definition \ref{def-ideal} has a finite set of
generators $\{g_1, \ldots, g_\ell\}$, i.e., for all $g \in
{\mathcal I}$, there exist $r_1, \ldots , r_\ell \in {\mathbb
R}[p,\zeta]$ with $g = r_1g_1 + \ldots + r_\ell g_\ell$. In such a
case, we write
\begin{equation*}
{\mathcal I} = {\mathrm{Ideal}}(g_1, \ldots, g_\ell) \, .
\end{equation*}

Let $A$ be a non-negative integer matrix with $K$ rows and $d$
columns.

\begin{definition}[Toric model]
The toric model associated to $A$ is the set of probability
distributions on $\{1, \ldots, K\}$ satisfying
\begin{equation*}
p_k = \zeta_0 \prod_{r=1}^d \zeta_r^{A_{k,r}}
\end{equation*}
for all $k=1, \ldots, K$.
\end{definition}
In the definition above, the parameter $\zeta_0$ acts as a
normalizing constant. As noticed in Section \ref{recall-sect}, a
toric model is the extension of a log-linear model and the matrix
$A$ is the matrix representation of the minimal sufficient
statistics.

Now, define the ideal ${\mathcal J}_A$ as the ideal generated by
the set of binomials
\begin{equation*}
\left\{ p_k - \prod_{r=1}^d \zeta_r^{A_{k,r}} \ : \ k=1, \ldots, K
\right\} \, .
\end{equation*}
Eliminating the $\zeta$ parameters, i.e., intersecting the ideal
${\mathcal J}_A$ with the polynomial ring ${\mathbb R}[p]\subset
{\mathbb R}[p,\zeta]$, we define the toric ideal associated to
$A$.

\begin{definition}
The toric ideal ${\mathcal I}_A$ associated to $A$ is
\begin{equation} \label{toric-id}
{\mathcal I}_A = \mathrm{Elim}(\zeta ,J_A) ={\mathcal J}_A \cap
{\mathbb R}[p] \, .
\end{equation}
\end{definition}

It is known that the toric ideal in Eq. \eqref{toric-id} is
generated by a finite set of pure homogeneous binomials $\{b_1 ,
\ldots, b_\ell\}$. To actually compute a set of generators of
${\mathcal I}_A$ one can use Computer Algebra softwares such as
CoCoA together with the command {\tt Elim} \citep{cocoa}. For
toric ideals, specific algorithms are implemented in {\tt 4ti2}
\citep{4ti2}.

The toric ideal ${\mathcal I}_A$ has two major meanings in
Algebraic Statistics. From the combinatorial side, the binomials
$b_1, \ldots, b_\ell$ specify a Markov basis for the statistical
model, while from a geometric point of view they describe the
statistical model.

\begin{definition}
Let $f \in {\mathbb N}^K$ be a contingency table with $K$ cells,
and let $A$ be a $K \times d$ matrix. The reference set of $f$
under $A$ is:
\begin{equation*}
{\mathcal F}_A(f) = \left\{ f' \in {\mathbb N}^k \ : \ A^t f' = A^t f
\right\} \, .
\end{equation*}
\end{definition}

\begin{definition}[Markov basis]
A set of tables ${\mathcal M}_A = \{m_1, \ldots, m_\ell  \}$, $m_j
\in {\mathbb Z}^K$, is a Markov basis for the reference set
${\mathcal F}_A(f)$ if $A^t m_j=0$ for all $j$, and for any pair
of tables $f', f'' \in {\mathcal F}_A(f)$ there exist a sequence
of moves $(m_{j_1}, \ldots, m_{j_W})$ and a sequence of signs
$(\epsilon_i)_{i=1}^W$ with $\epsilon_i = \pm 1$ such that
\begin{equation*}
f'' = f' + \sum_{i=1}^W \epsilon_i m_{j_i} \ \ \ \ {\mbox and } \
\ \ \ f' + \sum_{i=1}^w \epsilon_i m_{j_i} \geq 0
\end{equation*}
for all $1 \leq w \leq W$. The elements of a Markov basis are
called moves.
\end{definition}

\begin{definition}[log-vector] \label{logvector}
Given a binomial in ${\mathbb R}[p]$
\begin{equation*}
b = \prod_{k=1}^K p_k^{m^+(k)} - \prod_{k=1}^K p_k^{m^-(k)} \, ,
\end{equation*}
its log-vector is
\begin{equation*}
m = m^+ - m^- \in {\mathbb Z}^K \, .
\end{equation*}
\end{definition}

\begin{theorem}[Diaconis-Sturmfels] \label{DS-teo}
A set of vectors $\{m_1, \ldots , m_\ell\}$ is a Markov basis for
the toric model associated to $A$ if and only if the corresponding
binomials $b_1, \ldots, b_\ell$ generate the toric ideal
${\mathcal I}_A$.
\end{theorem}

Now, we show how the toric ideal ${\mathcal I}_A$ identifies the
statistical toric model.

\begin{definition}
The set of points
\begin{equation*}
{\mathcal V}_A = \left\{ p=(p_1, \ldots, p_K) \ : \ g(p) = 0 \
\mbox{ for all } \ g \in {\mathcal I}_A \right\}
\end{equation*}
is the variety associated to $A$.
\end{definition}

To actually determine the variety ${\mathcal V}_A$, it is enough
to solve the polynomial system $b_1(p) = 0 , \ldots, b_\ell(p)=0$,
where $b_1, \ldots, b_\ell$ is a system of generators of
${\mathcal I}_A$.

The relations between the ideal ${\mathcal I}_A$ and the variety
${\mathcal V}_A$ imply that a unique computational algorithm
produces both the Markov basis and the equations defining the
variety. Moreover, the following fundamental result holds.

\begin{proposition} \label{dualincl}
Let ${\mathcal I}_{A_1}$ and ${\mathcal I}_{A_2}$ be two toric
ideals. Then:
\begin{equation*}
{\mathcal I}_{A_1} \subset {\mathcal I}_{A_2} \
\Longleftrightarrow {\mathcal V}_{A_2} \subset {\mathcal V}_{A_1}
\end{equation*}
\end{proposition}

Finally, the statistical toric model is formed by the probability
distributions in ${\mathcal V}_A$, i.e., the statistical toric
model is simply ${\mathcal V}_A \cap \Delta$.

\bibliographystyle{decsci}
\bibliography{tuttopm}

\end{document}